\documentclass[11pt]{article}
\usepackage[a4paper,margin=1in]{geometry}
\usepackage[T1]{fontenc}
\usepackage[utf8]{inputenc}
\usepackage{lmodern}
\usepackage{microtype}
\usepackage{amsmath,amssymb,amsthm,mathtools}
\usepackage{xcolor}
\usepackage{graphicx}
\usepackage{tikz}
\usetikzlibrary{arrows.meta,positioning,calc,fit,backgrounds,patterns,shapes.geometric}
\usepackage[font=small,labelfont=bf]{caption}
\usepackage{subcaption}
\usepackage{enumitem}
\usepackage{float}
\usepackage[hidelinks,pdfusetitle]{hyperref}

\newcommand{\AuthorName}{Vadim E. Levit and Eugen Mandrescu}
\newcommand{\AuthorOne}{Vadim E. Levit}
\newcommand{\AffiliationOne}{Department of Mathematics\\ Ariel University, Israel}
\newcommand{\EmailOne}{levitv@ariel.ac.il}
\newcommand{\AuthorTwo}{Eugen Mandrescu}
\newcommand{\AffiliationTwo}{Department of Computer Science\\ Holon Institute of Technology, Israel}
\newcommand{\EmailTwo}{eugen\_m@hit.ac.il}

\hypersetup{
  pdftitle={The family of all local maximum independent sets is an augmentoid},
  pdfsubject={On Problem 5.5 of Levit and Mandrescu},
  pdfkeywords={augmentoid, local maximum independent set, crown, greedoid, graph theory},
  pdfauthor={\AuthorName}
}

\newtheorem{theorem}{Theorem}[section]
\newtheorem{lemma}[theorem]{Lemma}

\newtheorem{proposition}[theorem]{Proposition}
\newtheorem{corollary}[theorem]{Corollary}
\newtheorem{example}[theorem]{Example}
\newtheorem{problem}[theorem]{Problem}

\newcommand{\Crown}{\mathrm{Crown}}
\newcommand{\CritIndep}{\mathrm{CritIndep}}

\newcommand{\core}{\mathrm{core}}
\newcommand{\corona}{\mathrm{corona}}
\newcommand{\diff}{d}

\tikzset{
  conceptbox/.style={rounded corners=2.5pt, draw=black, line width=0.8pt,
    fill=black!3, align=center, inner sep=5pt, minimum width=3.2cm, minimum height=1.3cm},
  conceptnote/.style={align=center, font=\footnotesize},
  conceptarrow/.style={-{Latex[length=2.2mm]}, line width=0.9pt},
  vertex/.style={circle, draw=black, line width=0.9pt, fill=white,
    minimum size=7.2mm, inner sep=0pt, font=\small},
  sonly/.style={vertex, fill=black!15},
  tonly/.style={vertex, double=black, double distance=0.7pt},
  bothset/.style={vertex, fill=black!70, text=white},
  added/.style={vertex, fill=black!35},
  focus/.style={draw=black!45, rounded corners=5pt, fill=black!5,
    inner xsep=5pt, inner ysep=4pt},
  edge/.style={line width=0.95pt, shorten >=3.2pt, shorten <=3.2pt, line cap=round}
}

\title{The family of all local maximum independent sets is an augmentoid}
\author{%
  \begin{tabular}[t]{@{}c@{\hspace{3em}}c@{}}
    \begin{tabular}[t]{@{}c@{}}
      \AuthorOne\\[2pt]
      \small \AffiliationOne\\
      \small\texttt{\EmailOne}
    \end{tabular}
    &
    \begin{tabular}[t]{@{}c@{}}
      \AuthorTwo\\[2pt]
      \small \AffiliationTwo\\
      \small\texttt{\EmailTwo}
    \end{tabular}
  \end{tabular}%
}
\date{}

\begin{document}
\maketitle

\begin{abstract}
It was proved in \cite{LM2022} that both $(V(G),\Crown(G))$ and $(V(G),\CritIndep(G))$ are augmentoids, established partial augmentation phenomena for the family $\Psi(G)$ of local maximum independent sets, and asked in Problem~5.5 to characterize the graphs whose family $\Psi(G)$ is an augmentoid. We prove that the answer is positive in full generality: for every finite simple graph $G$, the set system $(V(G),\Psi(G))$ is an augmentoid. The proof is constructive. If $S,T\in\Psi(G)$, then the explicit choice
\[
A=S \setminus N[T],\qquad B=T \setminus N[S]
\]
satisfies
\[
T\cup A\in\Psi(G),\qquad S\cup B\in\Psi(G),\qquad |T\cup A|=|S\cup B|.
\]
As a structural consequence, for every fixed $S\in\Psi(G)$ the map $T\mapsto S\cup T$ induces a canonical bijection from $\Psi(G-N[S])$ onto the members of $\Psi(G)$ containing $S$, and
\[
\alpha(G)=|S|+\alpha(G-N[S]).
\]
This decomposition also yields explicit formulas for the intersection and the union of all the maximum independent sets extending $S$, together with counting formulas for the local maximum and maximum independent sets containing $S$. We also add a short visual guide to the framework $\CritIndep(G)\subseteq\Crown(G)\subseteq\Psi(G)$ and end with several natural follow-up problems suggested by the theorem.
\end{abstract}



\section{Introduction}
Throughout, $G$ is a finite simple graph with vertex set $V(G)$ and edge set $E(G)$. For $X\subseteq V(G)$, the subgraph induced by $X$ is denoted by $G[X]$, while $N(X)$ and $N[X]=X\cup N(X)$ stand for the open and closed neighborhoods of $X$, respectively. A set $S\subseteq V(G)$ is \emph{independent} (or \emph{stable}) if no two vertices of $S$ are adjacent. We write $\alpha(G)$ for the maximum cardinality of an independent set in $G$ and $\Omega(G)$ for the family of all maximum independent sets.

A matching in $G$ is a set of pairwise disjoint edges; as usual, $\mu(G)$ denotes the size of a maximum matching. Recall that $G$ is a \emph{K\H{o}nig--Egerv\'ary graph} whenever
\[
\alpha(G)+\mu(G)=|V(G)|.
\]
We also use the following notations
\[
\core(G):=\bigcap \Omega(G),\qquad \corona(G):=\bigcup \Omega(G)
\]
for the core and corona of $G$.
For a set $X\subseteq V(G)$, put $\diff(X)=|X|-|N(X)|$. An independent set $S$ is \emph{critical} if
\[
\diff(S)=d(G), \text{ where } d(G)=\max\{\diff(I): I\text{ is an independent set of }G\}.
\]
We denote the family of all critical independent sets by $\CritIndep(G)$. Following~\cite{LM2022}, an independent set $S$ is a \emph{crown} if there exists a matching from $N(S)$ into $S$; the family of all crowns is denoted by $\Crown(G)$.

A set $S\subseteq V(G)$ is a \emph{local maximum independent set} if it is a maximum independent set of the induced subgraph $G[N[S]]$. The family of all such sets is written as $\Psi(G)$. The importance of $\Psi(G)$ goes back to the Nemhauser--Trotter theorem~\cite{NT1975}: every member of $\Psi(G)$ can be extended to a member of $\Omega(G)$. Levit and Mandrescu initiated the systematic study of these families in~\cite{LM2002}, where they proved that $\Psi(T)$ is a greedoid for every forest $T$. They then analyzed further graph classes in which the local maximum independent sets form greedoids, including bipartite graphs with uniquely restricted maximum matchings~\cite{LM2003}, triangle-free graphs with uniquely restricted maximum matchings~\cite{LM2007}, well-covered graphs~\cite{LM2008wc}, very well-covered graphs~\cite{LM2011}, and several graph operations preserving greedoid structure~\cite{LM2010}. A broader structural study culminated in the general criterion that if $\Psi(G)$ satisfies accessibility, then it is in fact an interval greedoid~\cite{LM2008ig,LM2012}.

The 2022 paper~\cite{LM2022} places these ideas into a larger framework built from critical independent sets, crowns, and local maximum independent sets. In particular, the following inclusions were shown:
\[
\CritIndep(G)\subseteq\Crown(G)\subseteq\Psi(G).
\]
Moreover, it was proved that $(V(G),\CritIndep(G))$ and $(V(G),\Crown(G))$ are augmentoids, and two partial augmentation results for $\Psi(G)$ were established: one for independent unions of disjoint members of $\Psi(G)$ and one for pairs with nested closed neighborhoods~\cite[Proposition~3.20, Theorem~5.3, Corollary~5.4]{LM2022}. The same paper also established several equality results that are directly relevant here. Namely,
\begin{align*}
\Crown(G)=\Psi(G)
&\Longleftrightarrow G[N[S]]\text{ is a K\H{o}nig--Egerv\'ary graph}\\
&\qquad\text{for every }S\in\Psi(G),
\end{align*}
while
\begin{align*}
\CritIndep(G)=\Crown(G)=\Psi(G)
&\Longleftrightarrow G[N[S]]\text{ is a K\H{o}nig--Egerv\'ary graph}\\
&\qquad\text{with a perfect matching for every }S\in\Psi(G)
\end{align*}
by~\cite[Proposition~4.1 and Theorem~4.13]{LM2022}. In particular, $\Crown(G)=\Psi(G)$ for bipartite graphs and for very well-covered graphs~\cite[Corollary~4.2 and Proposition~4.4]{LM2022}, while $\CritIndep(G)=\Crown(G)$ was already proved in~\cite[Proposition~4.7]{LM2022} for every K\H{o}nig--Egerv\'ary graph with a perfect matching. The same paper also records the implication
\[
\CritIndep(G)=\Crown(G)\Longrightarrow d(G)=0,
\]
and formulates as Problem~5.3 the task of characterizing the graphs satisfying $\CritIndep(G)=\Crown(G)$~\cite[p.~494]{LM2022}. Special cases obtained there include the bipartite characterization
\[
\CritIndep(G)=\Crown(G)=\Psi(G)
\quad\text{if and only if}\quad
G\text{ has a perfect matching},
\]
and, for trees of order at least two, the equivalences
\begin{align*}
\CritIndep(T)=\Crown(T)
&\Longleftrightarrow \CritIndep(T)=\Psi(T)\\
&\Longleftrightarrow d(T)=0\\
&\Longleftrightarrow T\text{ has a perfect matching}
\end{align*}
from~\cite[Corollary~4.9 and Corollary~4.12]{LM2022}. Very recently, Pereyra~\cite[Theorem~3.6]{Pereyra2026} resolved exactly Problem~5.3 by proving the converse implication, and hence
\[
\CritIndep(G)=\Crown(G)\quad\text{if and only if}\quad d(G)=0.
\]
He also gave additional reformulations of $\Crown(G)=\Psi(G)$ and $\CritIndep(G)=\Crown(G)=\Psi(G)$~\cite[Theorems~3.13 and~3.14]{Pereyra2026}. This motivates the following.

\begin{problem}
\cite[Problem~5.5]{LM2022}\quad Characterize graphs whose families of local maximum independent sets are augmentoids.
\end{problem}

Following~\cite{LM2022}, a pair $(E,\mathcal F)$ with nonempty $\mathcal F\subseteq 2^E$ is called an \emph{augmentoid} if for every $X,Y\in\mathcal F$ there exist sets $A\subseteq X \setminus Y$ and $B\subseteq Y \setminus X$ such that
\[
Y\cup A\in\mathcal F,\qquad X\cup B\in\mathcal F,\qquad |Y\cup A|=|X\cup B|.
\]
The purpose of the present note is to show that no restriction on $G$ is needed.

\begin{theorem}\label{thm:main}
For every graph $G$, the set system $(V(G),\Psi(G))$ is an augmentoid.
\end{theorem}

Equivalently, for every $S,T\in\Psi(G)$ there exist sets
\[
A\subseteq S \setminus T,\qquad B\subseteq T \setminus S
\]
such that $T\cup A\in\Psi(G)$, $S\cup B\in\Psi(G)$, and $|T\cup A|=|S\cup B|$. In fact, we show that one may always choose
\[
A=S \setminus N[T],\qquad B=T \setminus N[S].
\]
A second structural consequence, isolated in Proposition~\ref{prop:decomp}, is a closed-neighborhood decomposition: $\alpha(G)=|S|+\alpha(H)$ for every fixed $S\in\Psi(G)$, and writing $H:=G-N[S]$, the function
\[
T\longmapsto S\cup T
\]
induces a bijection from $\Psi(H)$ onto the members of $\Psi(G)$ that contain $S$.

Before proving Theorem~\ref{thm:main}, we include a brief visual guide to the general framework and a concrete example of the canonical augmentation.

\section{Examples and figures}

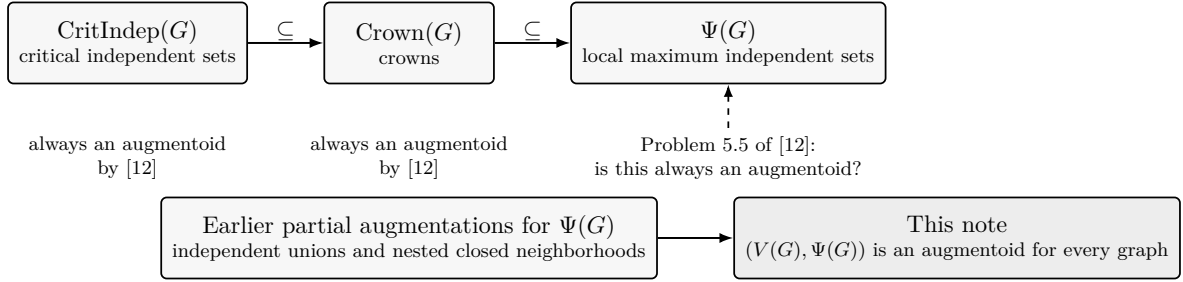
\begin{figure}[t]
\centering
\resizebox{0.97\textwidth}{!}{%
\begin{tikzpicture}[node distance=10mm and 12mm]
  \node[conceptbox, minimum width=3.0cm] (crit) {$\CritIndep(G)$\\[-1mm]\footnotesize critical independent sets};
  \node[conceptbox, right=of crit, minimum width=2.7cm] (crown) {$\Crown(G)$\\[-1mm]\footnotesize crowns};
  \node[conceptbox, right=of crown, minimum width=3.5cm] (psi) {$\Psi(G)$\\[-1mm]\footnotesize local maximum independent sets};

  \draw[conceptarrow] (crit) -- node[above, fill=white, inner sep=1pt] {$\subseteq$} (crown);
  \draw[conceptarrow] (crown) -- node[above, fill=white, inner sep=1pt] {$\subseteq$} (psi);

  \node[conceptnote, below=7mm of crit] (critnote) {always an augmentoid\\by~\cite{LM2022}};
  \node[conceptnote, below=7mm of crown] (crownnote) {always an augmentoid\\by~\cite{LM2022}};
  \node[conceptnote, below=7mm of psi] (psinote) {Problem~5.5 of~\cite{LM2022}:\\is this always an augmentoid?};

  \draw[dashed, -{Latex[length=2mm]}, line width=0.8pt] (psinote.north) -- (psi.south);

  \node[conceptbox, below=18mm of crown, minimum width=5.6cm] (partial)
    {Earlier partial augmentations for $\Psi(G)$\\[-1mm]
     \footnotesize independent unions and nested closed neighborhoods};
  \node[conceptbox, right=12mm of partial, fill=black!7, minimum width=4.0cm] (answer)
    {This note\\[-1mm]\footnotesize $(V(G),\Psi(G))$ is an augmentoid for every graph};
  \draw[conceptarrow] (partial) -- (answer);
\end{tikzpicture}%
}
\caption{A roadmap of the framework developed in~\cite{LM2022} and of the present result. The theorem proved below upgrades the two earlier augmentation mechanisms for $\Psi(G)$ to a full augmentoid statement valid for all graphs.}
\label{fig:roadmap}
\end{figure}

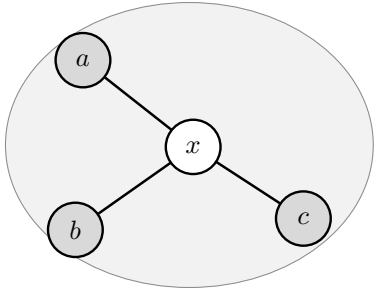
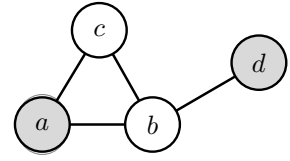
\begin{figure}[H]
\centering
\begin{subfigure}[t]{0.47\textwidth}
\centering
\begin{tikzpicture}[scale=1]
  \coordinate (x) at (0,0);
  \coordinate (a) at (-1.45,1.15);
  \coordinate (b) at (-1.55,-1.1);
  \coordinate (c) at (1.45,-0.95);
  \draw[edge] (x) -- (a);
  \draw[edge] (x) -- (b);
  \draw[edge] (x) -- (c);
  \begin{scope}[on background layer]
    \node[focus, fit=(a)(b)(c), ellipse, inner xsep=6pt, inner ysep=6pt] {};
  \end{scope}
  \node[vertex] at (x) {$x$};
  \node[sonly] at (a) {$a$};
  \node[sonly] at (b) {$b$};
  \node[sonly] at (c) {$c$};
\end{tikzpicture}
\caption{A star $K_{1,3}$.}
\label{fig:star}
\end{subfigure}\hfill
\begin{subfigure}[t]{0.47\textwidth}
\centering
\begin{tikzpicture}[scale=1]
  \coordinate (a) at (-1.3,0);
  \coordinate (b) at (0.15,0);
  \coordinate (c) at (-0.55,1.25);
  \coordinate (d) at (1.55,0.82);
  \draw[edge] (a) -- (b) -- (c) -- (a);
  \draw[edge] (b) -- (d);
  \begin{scope}[on background layer]
    \node[focus, fit=(a), ellipse, inner xsep=7pt, inner ysep=8pt] {};
  \end{scope}
  \node[sonly] at (a) {$a$};
  \node[vertex] at (b) {$b$};
  \node[vertex] at (c) {$c$};
  \node[sonly] at (d) {$d$};
\end{tikzpicture}
\caption{A triangle with a pendant vertex.}
\label{fig:triangle}
\end{subfigure}
\caption{Two small examples separating the families $\CritIndep(G)$, $\Crown(G)$, and $\Psi(G)$. The star from Figure~\ref{fig:star} witnesses the failure of the equality $\CritIndep(G)=\Crown(G)$, while the graph from Figure~\ref{fig:triangle} witnesses the failure of the equality $\Crown(G)=\Psi(G)$.}
\label{fig:separation}
\end{figure}

\begin{figure}[H]
\centering
\begin{subfigure}[t]{0.31\textwidth}
\centering
\begin{tikzpicture}[scale=0.96]
  \coordinate (a) at (-1.3,0);
  \coordinate (b) at (-2.55,0);
  \coordinate (c) at (0,0);
  \coordinate (d) at (0,1.45);
  \coordinate (e) at (1.35,1.45);
  \coordinate (f) at (-1.35,1.45);
  \draw[edge] (b)--(a)--(c);
  \draw[edge] (c)--(d);
  \draw[edge] (c)--(e);
  \draw[edge] (c)--(f);
  \node[tonly] at (b) {$b$};
  \node[sonly] at (a) {$a$};
  \node[vertex] at (c) {$c$};
  \node[bothset] at (d) {$d$};
  \node[sonly] at (e) {$e$};
  \node[tonly] at (f) {$f$};
\end{tikzpicture}
\caption{The pair $S,T$.}
\label{fig:aug-a}
\end{subfigure}\hfill
\begin{subfigure}[t]{0.31\textwidth}
\centering
\begin{tikzpicture}[scale=0.96]
  \coordinate (a) at (-1.3,0);
  \coordinate (b) at (-2.55,0);
  \coordinate (c) at (0,0);
  \coordinate (d) at (0,1.45);
  \coordinate (e) at (1.35,1.45);
  \coordinate (f) at (-1.35,1.45);
  \draw[edge] (b)--(a)--(c);
  \draw[edge] (c)--(d);
  \draw[edge] (c)--(e);
  \draw[edge] (c)--(f);
  \node[vertex] at (b) {$b$};
  \node[sonly] at (a) {$a$};
  \node[vertex] at (c) {$c$};
  \node[bothset] at (d) {$d$};
  \node[sonly] at (e) {$e$};
  \node[added] at (f) {$f$};
\end{tikzpicture}
\caption{$S^+=S\cup(T\setminus N[S])$.}
\label{fig:aug-b}
\end{subfigure}\hfill
\begin{subfigure}[t]{0.31\textwidth}
\centering
\begin{tikzpicture}[scale=0.96]
  \coordinate (a) at (-1.3,0);
  \coordinate (b) at (-2.55,0);
  \coordinate (c) at (0,0);
  \coordinate (d) at (0,1.45);
  \coordinate (e) at (1.35,1.45);
  \coordinate (f) at (-1.35,1.45);
  \draw[edge] (b)--(a)--(c);
  \draw[edge] (c)--(d);
  \draw[edge] (c)--(e);
  \draw[edge] (c)--(f);
  \node[tonly] at (b) {$b$};
  \node[vertex] at (a) {$a$};
  \node[vertex] at (c) {$c$};
  \node[bothset] at (d) {$d$};
  \node[added] at (e) {$e$};
  \node[tonly] at (f) {$f$};
\end{tikzpicture}
\caption{$T^+=T\cup(S - N[T])$.}
\label{fig:aug-c}
\end{subfigure}
\caption{A concrete augmentation for Theorem~\ref{thm:main}. In panel (a), gray vertices belong to $S - T$, double-circled vertices belong to $T - S$, and the black vertex belongs to $S\cap T$. Panels (b) and (c) show the two augmented local maximum independent sets produced by the theorem.}
\label{fig:augmentation}
\end{figure}
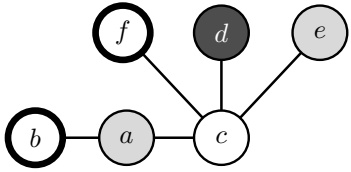
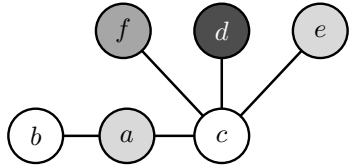
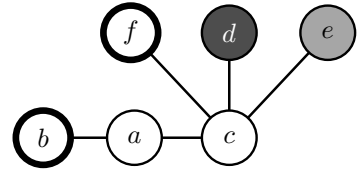

\begin{example}\label{ex:star}
Let $G$ be the star from Figure~\ref{fig:star}. Then $\{a,b,c\}$ is the unique critical independent set, because
\[
\diff(\{a,b,c\})=3-1=2
\]
and every proper independent subset of $\{a,b,c\}$ has smaller difference. Thus $d(G)=2$, and the characterization of Pereyra~\cite[Theorem~3.6]{Pereyra2026} yields
\[
\CritIndep(G)\neq\Crown(G).
\]
On the other hand, every subset of $\{a,b,c\}$ is a crown: its neighborhood is either empty or the singleton $\{x\}$, and hence can be matched into the set itself. Since every crown belongs to $\Psi(G)$ by~\cite{LM2022}, we obtain
\[
\CritIndep(G)\subsetneq\Crown(G)=\Psi(G).
\]
\end{example}

\begin{example}\label{ex:triangle}
Let $G$ be the graph from Figure~\ref{fig:triangle}. The singleton $\{a\}$ is a local maximum independent set, because $G[N[\{a\}]]$ is the triangle on $\{a,b,c\}$, and $\{a\}$ is a maximum independent set of that triangle. However, $\{a\}$ is not a crown: its neighborhood is $\{b,c\}$, which cannot be matched into a singleton. A direct inspection shows that
\[
\Crown(G)=\CritIndep(G)=\{\varnothing,\{d\},\{a,d\},\{c,d\}\},
\]
and each of these sets has difference $0$, so $d(G)=0$. Thus this graph witnesses that the left equality $\CritIndep(G)=\Crown(G)$ may hold while the second inclusion $\Crown(G)\subseteq\Psi(G)$ is still strict, even for a small connected non-bipartite K\H{o}nig--Egerv\'ary graph.
\end{example}

\begin{example}\label{ex:augmentation}
Consider the graph from Figure~\ref{fig:augmentation}. Let
\[
S=\{a,d,e\},\qquad T=\{b,d,f\}.
\]
Then $S,T\in\Psi(G)$. Indeed, $G[N[S]]$ has maximum independent sets $\{a,d,e\}$ and $\{b,d,e\}$, while $G[N[T]]$ has maximum independent sets $\{b,d,f\}$ and $\{a,d,f\}$. The canonical augmentation isolates the vertices outside the opposite closed neighborhoods:
\[
S - N[T]=\{e\},\qquad T - N[S]=\{f\}.
\]
Hence, Theorem~\ref{thm:main} produces
\[
S^+=S\cup(T - N[S])=\{a,d,e,f\}\in\Psi(G)
\]
and
\[
T^+=T\cup(S - N[T])=\{b,d,e,f\}\in\Psi(G),
\]
with $|S^+|=|T^+|=4$.
\end{example}

\begin{example}\label{ex:previous-results}
The theorem subsumes the two augmentation mechanisms from~\cite{LM2022}. If $S,T\in\Psi(G)$ are disjoint and $S\cup T$ is independent, then $S\cap N(T)=T\cap N(S)=\varnothing$, so the canonical choice gives
\[
S \setminus N[T]=S,\qquad T \setminus N[S]=T,
\]
and therefore $S\cup T\in\Psi(G)$, recovering~\cite[Proposition~3.20]{LM2022}. Likewise, if $N[S]\subseteq N[T]$, then $S \setminus N[T]=\varnothing$, so the theorem yields
\[
S\cup (T \setminus N[S])\in\Psi(G)
\quad\text{and}\quad
\bigl|S\cup (T \setminus N[S])\bigr|=|T|,
\]
which is precisely~\cite[Theorem~5.3 and Corollary~5.4]{LM2022}.
\end{example}

\section{Proof of Theorem~\ref{thm:main}}
\noindent We use the following classical lemma, often called Berge's Maximum Stable Set Lemma; see, for instance,~\cite[Theorem~1.2]{LM2008ig}. For the sake of completeness, we also include a short Hall-type proof.

\begin{lemma}\label{lem:matching}
Let $H$ be a graph, let $M$ be a maximum independent set of $H$, and let $I$ be any independent set of $H$. Then there exists a matching from $I \setminus M$ into $M \setminus I$.
\end{lemma}

\begin{proof}
Assume that no such matching exists. By Hall's theorem, there is a set $X\subseteq I \setminus M$ such that
\[
|N_H(X)\cap(M \setminus I)|<|X|.
\]
Since $I$ is independent, no vertex of $X$ has a neighbor in $M\cap I$. Therefore $(M \setminus N_H(X))\cup X$ is an independent set of $H$, and
\[
|(M \setminus N_H(X))\cup X|
 =|M|-|N_H(X)\cap M|+|X|
 =|M|-|N_H(X)\cap(M \setminus I)|+|X|>|M|,
\]
a contradiction.
\end{proof}

\begin{lemma}\label{lem:cross}
If $S,T\in\Psi(G)$, then there exists a perfect matching between $S\cap N(T)$ and $T\cap N(S)$. In particular,
\[
|S\cap N(T)|=|T\cap N(S)|.
\]
\end{lemma}

\begin{proof}
Consider the induced subgraph $H_S=G[N[S]]$. Since $S\in\Psi(G)$, the set $S$ is a maximum independent set of $H_S$. The set $T\cap N[S]$ is an independent set of $H_S$, so by Lemma~\ref{lem:matching} there exists a matching from
\[
(T\cap N[S]) \setminus S=T\cap N(S)
\]
into
\[
S \setminus (T\cap N[S])=S \setminus T.
\]
Every vertex in the image is adjacent to a vertex of $T\cap N(S)$, hence belongs to $S\cap N(T)$. Thus there exists a matching from $T\cap N(S)$ into $S\cap N(T)$. By symmetry, there is also a matching from $S\cap N(T)$ into $T\cap N(S)$. The two finite sets therefore have the same cardinality, and either matching is perfect.
\end{proof}

Now fix $S,T\in\Psi(G)$ and define
\[
S_0:=S \setminus N[T],\qquad T_0:=T \setminus N[S],
\]
\[
S^+:=S\cup T_0,\qquad T^+:=T\cup S_0.
\]
Since $T_0\cap N[S]=\varnothing$ and $S_0\cap N[T]=\varnothing$, the sets $S^+$ and $T^+$ are independent.

\begin{lemma}\label{lem:outside}
First, $T_0\in\Psi(G-N[S])$. Second, symmetrically, $S_0\in\Psi(G-N[T])$.
\end{lemma}

\begin{proof}
We prove the first statement; the second is symmetric. Set
\[
X:=S\cap N(T),\qquad Y:=T\cap N(S),\qquad C:=S\cap T.
\]
By Lemma~\ref{lem:cross}, we have $|X|=|Y|$.

Let $H:=G-N[S]$, and let $R$ be any independent set of $H[N_H[T_0]]$. Then $R\subseteq N[T_0] - N[S]\subseteq N[T] - N[S]$, so $R$ is anticomplete to $S$. Consequently,
\[
R\cup X\cup C
\]
is an independent set of $G[N[T]]$. Since $T\in\Psi(G)$, the set $T$ is a maximum independent set of $G[N[T]]$. Hence,
\[
|R|+|X|+|C|\le |T|=|T_0|+|Y|+|C|=|T_0|+|X|+|C|,
\]
and therefore $|R|\le |T_0|$. This proves that $T_0$ is a maximum independent set of $H[N_H[T_0]]$, i.e. $T_0\in\Psi(H)=\Psi(G-N[S])$.
\end{proof}

\begin{lemma}\label{lem:plus}
$S^+,T^+\in\Psi(G)$.
\end{lemma}

\begin{proof}
We prove that $S^+\in\Psi(G)$; the proof for $T^+$ is symmetric. Let $J$ be an independent set of $G[N[S^+]]$. Put
\[
J_1:=J\cap N[S],\qquad J_2:=J\setminus N[S].
\]
Then $J_1$ is an independent set of $G[N[S]]$, so $|J_1|\le |S|$ because $S\in\Psi(G)$.

Let $H:=G-N[S]$. Since $N[S^+]=N[S]\cup N[T_0]$, we have
\[
J_2\subseteq N[T_0] \setminus N[S]=N_H[T_0].
\]
Moreover, $J_2$ is an independent set of $H[N_H[T_0]]$. By Lemma~\ref{lem:outside}, the set $T_0$ belongs to $\Psi(H)$, hence $|J_2|\le |T_0|$. Therefore,
\[
|J|=|J_1|+|J_2|\le |S|+|T_0|=|S^+|.
\]
Thus, $S^+$ is a maximum independent set of $G[N[S^+]]$, and so $S^+\in\Psi(G)$.
\end{proof}

\begin{lemma}\label{lem:same-size}
$|S^+|=|T^+|$.
\end{lemma}

\begin{proof}
We have the disjoint decompositions
\[
S=(S \setminus N[T])\cup(S\cap N(T))\cup(S\cap T)
\]
and
\[
T=(T \setminus N[S])\cup(T\cap N(S))\cup(S\cap T).
\]
Therefore,
\[
|S^+|=|S|+|T \setminus N[S]|
      =|S \setminus N[T]|+|S\cap N(T)|+|S\cap T|+|T \setminus N[S]|
\]
and likewise
\[
|T^+|=|T|+|S \setminus N[T]|
      =|T \setminus N[S]|+|T\cap N(S)|+|S\cap T|+|S \setminus N[T]|.
\]
By Lemma~\ref{lem:cross}, the middle terms have the same cardinality, and so $|S^+|=|T^+|$.
\end{proof}

\begin{proof}[Proof of Theorem~\ref{thm:main}]
The family $\Psi(G)$ is nonempty because $\varnothing\in\Psi(G)$. Let $S,T\in\Psi(G)$. By Lemma~\ref{lem:plus},
\[
T\cup(S \setminus N[T])=T^+\in\Psi(G)
\]
and
\[
S\cup(T \setminus N[S])=S^+\in\Psi(G).
\]
Moreover,
\[
S \setminus N[T]\subseteq S \setminus T,\qquad T \setminus N[S]\subseteq T \setminus S,
\]
and Lemma~\ref{lem:same-size} gives $|T^+|=|S^+|$. Hence, the augmentation property holds for $\Psi(G)$ with the explicit choice
\[
A=S \setminus N[T],\qquad B=T \setminus N[S].
\]
Therefore, $(V(G),\Psi(G))$ is an augmentoid.
\end{proof}

\subsection*{A closed-neighborhood decomposition for local maximum independent sets}

The main theorem has a useful structural companion that will likely be more suitable for later recursive arguments.

\begin{proposition}[Closed-neighborhood decomposition]\label{prop:decomp}
Let $S\in\Psi(G)$ and $H:=G-N[S]$. Define
\[
\Psi_S(G):=\{U\in\Psi(G): S\subseteq U\},
\qquad
\Omega_S(G):=\{M\in\Omega(G): S\subseteq M\}.
\]
Then the following hold.
\begin{enumerate}[label=\textup{(\roman*)}, leftmargin=8mm]
\item $\alpha(G)=|S|+\alpha(H)$.

\item The map $\Phi_S:\Psi(H)\longrightarrow \Psi_S(G), \Phi_S(T)=S\cup T$ is a bijection.

\item $\Omega_S(G)=\{S\cup Q: Q\in\Omega(H)\}$. Equivalently, the maximum independent sets of $G$ containing $S$ are exactly the sets of the form $S\cup Q$ with $Q\in\Omega(H)$.

\item $\bigcap \Omega_S(G)=S\cup \core(H)\  \textup{and} \ \bigcup \Omega_S(G)=S\cup \corona(H).$

\end{enumerate}
\end{proposition}

\begin{proof}
Let $I$ be any independent set of $G$. Then $I\cap N[S]$ is an independent set of $G[N[S]]$, while $I \setminus N[S]$ is an independent set of $H$. Since $S\in\Psi(G)$, the set $S$ is a maximum independent set of $G[N[S]]$. Hence,
\[
|I|=|I\cap N[S]|+|I \setminus N[S]|\le |S|+\alpha(H).
\]
On the other hand, if $Q\in\Omega(H)$, then $S\cup Q$ is independent in $G$, so
\[
\alpha(G)\ge |S|+|Q|=|S|+\alpha(H).
\]
This proves \textup{(i)}.

To prove \textup{(ii)}, first let $T\in\Psi(H)$. Since $T\subseteq V(H)=V(G)\setminus N[S]$, the set $T$ is anticomplete to $S$, and therefore $S\cup T$ is independent in $G$. Let $J$ be an independent set of $G[N[S\cup T]]$, and put
\[
J_1:=J\cap N[S],\qquad J_2:=J \setminus N[S].
\]
Then $J_1$ is an independent set of $G[N[S]]$, and thus $|J_1|\le |S|$.
Also,
\[
J_2\subseteq N[S\cup T] \setminus N[S]=N_H[T],
\]
and $J_2$ is an independent set of $H[N_H[T]]$. Since $T\in\Psi(H)$, we get
$|J_2|\le |T|$. Therefore
\[
|J|=|J_1|+|J_2|\le |S|+|T|=|S\cup T|,
\]
which shows that $S\cup T\in\Psi(G)$.

Conversely, let $U\in\Psi_S(G)$, and write
\[
T:=U \setminus S.
\]
Since $U$ is independent and contains $S$, every vertex of $T$ lies outside $N[S]$, so $T\subseteq V(H)$. Let $R$ be an independent set of $H[N_H[T]]$. Then
\[
R\subseteq N_H[T]\subseteq N[U] \setminus N[S],
\]
and $S\cup R$ is an independent set of $G[N[U]]$. Since $U\in\Psi(G)$,
\[
|S|+|R|=|S\cup R|\le |U|=|S|+|T|,
\]
and therefore $|R|\le |T|$. Hence, $T\in\Psi(H)$. So the map $\Phi_S$ is bijective, with inverse
\[
\Psi_S(G)\longrightarrow \Psi(H),\qquad U\longmapsto U \setminus S.
\]

Finally, \textup{(iii)} follows from \textup{(i)}. If $Q\in\Omega(H)$, then $S\cup Q$ is independent in $G$, and by \textup{(i)}
\[
|S\cup Q|=|S|+\alpha(H)=\alpha(G),
\]
so $S\cup Q\in\Omega_S(G)$. Conversely, if $M\in\Omega_S(G)$, then $M=S\cup(M \setminus S)$ with $M \setminus S\subseteq V(H)$, and
\[
|M \setminus S|=\alpha(G)-|S|=\alpha(H)
\]
by \textup{(i)}. Thus $M \setminus S\in\Omega(H)$, and so
$M=S\cup Q$ for some $Q\in\Omega(H)$.

Part~\textup{(iv)} now follows immediately from \textup{(iii)}. Since $S$ is contained in every member of $\Omega_S(G)$,
\[
\Omega_S(G)=\{S\cup Q:Q\in\Omega(H)\}
\]
induces
\[
\bigcap \Omega_S(G)=S\cup \bigcap \Omega(H)=S\cup \core(H)
\]
and
\[
\bigcup \Omega_S(G)=S\cup \bigcup \Omega(H)=S\cup \corona(H).
\]
\end{proof}

\begin{corollary}[Counting formulas]\label{cor:counting}
Let $S\in\Psi(G)$ and let $H:=G-N[S]$. Then
\[
|\Psi_S(G)|=|\Psi(H)|
\qquad\textup{and}\qquad
|\Omega_S(G)|=|\Omega(H)|.
\]
\end{corollary}

\begin{proof}
The first equality is immediate from the bijection
\[
\Phi_S:\Psi(H)\longrightarrow \Psi_S(G),\qquad \Phi_S(T)=S\cup T,
\]
of Proposition~\ref{prop:decomp}\textup{(ii)}. The second equality follows from
\[
\Omega_S(G)=\{S\cup Q:Q\in\Omega(H)\}
\]
given by Proposition~\ref{prop:decomp}\textup{(iii)}.
\end{proof}

\section{Conclusion and open problems}
Theorem~\ref{thm:main} gives the strongest possible answer to Problem~5.5 of~\cite{LM2022}: for every graph, its family of all local maximum independent sets is an augmentoid. Together with Proposition~\ref{prop:decomp} and Corollary~\ref{cor:counting}, it shows that the new theorem does more than solve an isolated problem. It upgrades the partial augmentation results from~\cite[Proposition~3.20, Theorem~5.3, Corollary~5.4]{LM2022}, extends the decomposition identity of~\cite[Lemma~3.14]{LM2022} from $\CritIndep(G)$ to the whole family $\Psi(G)$, and yields an explicit description of all local maximum independent sets and all maximum independent sets containing a prescribed local maximum independent set, together with their relative core/corona structures and their cardinalities.

For the comparison with earlier work, it is important to separate precisely what was already proved in~\cite{LM2022} from what was added later. The 2022 paper already characterized the equality $\Crown(G)=\Psi(G)$ via K\H{o}nig--Egerv\'ary neighborhoods and the triple equality $\CritIndep(G)=\Crown(G)=\Psi(G)$ via K\H{o}nig--Egerv\'ary neighborhoods with perfect matchings~\cite[Proposition~4.1 and Theorem~4.13]{LM2022}. It also proved several special cases, including $\Crown(G)=\Psi(G)$ for bipartite and for very well-covered graphs, $\CritIndep(G)=\Crown(G)$ for K\H{o}nig--Egerv\'ary graphs with a perfect matching, the bipartite criterion $\CritIndep(G)=\Crown(G)=\Psi(G)$ if and only if $G$ has a perfect matching, and the corresponding equivalences for trees of order at least two~\cite[Corollary~4.2, Proposition~4.4, Proposition~4.7, Corollary~4.9, Corollary~4.12]{LM2022}. In its conclusion,~\cite{LM2022} also observed the implication
\[
\CritIndep(G)=\Crown(G)\Longrightarrow d(G)=0,
\]
and posed as Problem~5.3 the task of characterizing the graphs satisfying $\CritIndep(G)=\Crown(G)$. Pereyra's recent work~\cite[Theorem~3.6]{Pereyra2026} resolves exactly that open problem by proving the converse implication, and thus the equivalence $\CritIndep(G)=\Crown(G)$ if and only if $d(G)=0$; his paper also gives additional reformulations of the already-characterized equalities $\Crown(G)=\Psi(G)$ and $\CritIndep(G)=\Crown(G)=\Psi(G)$~\cite[Theorems~3.13 and~3.14]{Pereyra2026}. Hence the present note should be viewed as completing a different part of the 2022 program, namely Problem~5.5.

From the structural point of view, Proposition~\ref{prop:decomp} suggests that local maximum independent sets admit a recursive closed-neighborhood decomposition analogous in spirit to the critical-independence decomposition of Larson~\cite{Larson2011}. This viewpoint also aligns with the recent open problems of Pereyra~\cite[Problems~4.3 and~4.4]{Pereyra2026}, who asked for descriptions of $\Psi(G)$ and $\Crown(G)$ in terms of the Larson decomposition. This motivates the following directions.

\begin{problem}
Develop a recursive ``$\Psi$-decomposition'' theory. For a chosen $S\in\Psi(G)$, Proposition~\ref{prop:decomp} splits the independence number as
\[
\alpha(G)=|S|+\alpha(G-N[S]).
\]
Study how repeated use of this identity interacts with accessibility, interval-greedoid structure, and classes of graphs where the successive local pieces are unique or canonical.
\end{problem}

\begin{problem}
Generalize the decomposition behind Proposition~\ref{prop:decomp} to the level of augmentoids. Which augmentoids admit a meaningful notion of ``closed neighborhood'' or ``deletable block'' for which a feasible set $X$ induces a rank decomposition analogous to
\[
r(E)=|X|+r(E - C(X))?
\]
A satisfactory answer could provide a more general counterpart of Larson-type decompositions in the augmentoid setting.
\end{problem}

\begin{problem}
Characterize the graphs for which the canonical augmentation of Theorem~\ref{thm:main} can be iterated along deletions to produce accessibility chains for all members of $\Psi(G)$. By~\cite{LM2008ig,LM2012}, this is precisely the threshold at which the universal augmentoid structure of $\Psi(G)$ may strengthen to an interval greedoid.
\end{problem}

\begin{problem}
Develop a weighted version of the theorem. If one calls an independent set $S$ weighted-local maximum, when it has maximum weight in $G[N[S]]$, does an augmentoid-type exchange survive for the resulting weighted family? Even partial positive results would connect the present theorem back to the optimization flavor of the Nemhauser--Trotter theorem.
\end{problem}

\begin{problem}
Find an operator-theoretic model for augmentoids that captures the families $\CritIndep(G)$, $\Crown(G)$, and $\Psi(G)$ on a common footing. The recent bridge between greedoids and violator spaces in~\cite{KL2024} suggests seeking a weaker closure or violator-type representation for the augmentoid setting.
\end{problem}

\begin{problem}
Exploit the canonical augmentation algorithmically. For instance, can the rule
\[
(S,T)\longmapsto \bigl(S\cup(T - N[S]),\ T\cup(S - N[T])\bigr)
\]
be used to enumerate local maximum independent sets more efficiently, or to extend a prescribed $S\in\Psi(G)$ to a maximum independent set while preserving local optimality certificates along the way?
\end{problem}


\begin{thebibliography}{99}\small

\bibitem{KL2024}
Y.~Kempner and V.~E.~Levit,
\emph{Greedoids and violator spaces},
\emph{Axioms} \textbf{13} (2024), no.~9, Article~633.

\bibitem{Larson2007}
C.~E.~Larson,
\emph{A note on critical independence reductions},
\emph{Bulletin of the Institute of Combinatorics and its Applications} \textbf{51} (2007), 34--46.

\bibitem{Larson2011}
C.~E.~Larson,
\emph{The critical independence number and an independence decomposition},
\emph{European Journal of Combinatorics} \textbf{32} (2011), no.~2, 294--300.

\bibitem{LM2002}
V.~E.~Levit and E.~Mandrescu,
\emph{A new greedoid: the family of local maximum stable sets of a forest},
\emph{Discrete Applied Mathematics} \textbf{124} (2002), no.~1--3, 91--101.

\bibitem{LM2003}
V.~E.~Levit and E.~Mandrescu,
\emph{Local maximum stable sets in bipartite graphs with uniquely restricted maximum matchings},
\emph{Discrete Applied Mathematics} \textbf{132} (2003), no.~1--3, 163--174.

\bibitem{LM2007}
V.~E.~Levit and E.~Mandrescu,
\emph{Triangle-free graphs with uniquely restricted maximum matchings and their corresponding greedoids},
\emph{Discrete Applied Mathematics} \textbf{155} (2007), no.~18, 2414--2425.

\bibitem{LM2008wc}
V.~E.~Levit and E.~Mandrescu,
\emph{Well-covered graphs and greedoids},
in J.~Harland and P.~Manyem (eds.), \emph{Theory of Computing 2008: Proceedings of the Fourteenth Computing: The Australasian Theory Symposium (CATS 2008)},
Conferences in Research and Practice in Information Technology, vol.~77, ACS, 2008, pp.~89--94.

\bibitem{LM2008ig}
V.~E.~Levit and E.~Mandrescu,
\emph{Interval greedoids and families of local maximum stable sets of graphs},
arXiv:0811.4089, 2008.

\bibitem{LM2010}
V.~E.~Levit and E.~Mandrescu,
\emph{Graph operations that are good for greedoids},
\emph{Discrete Applied Mathematics} \textbf{158} (2010), no.~13, 1418--1423.

\bibitem{LM2011}
V.~E.~Levit and E.~Mandrescu,
\emph{Very well-covered graphs of girth at least four and local maximum stable set greedoids},
\emph{Discrete Mathematics, Algorithms and Applications} \textbf{3} (2011), no.~2, 245--252.

\bibitem{LM2012}
V.~E.~Levit and E.~Mandrescu,
\emph{On local maximum stable set greedoids},
\emph{Discrete Mathematics} \textbf{312} (2012), no.~3, 588--596.

\bibitem{LM2022}
V.~E.~Levit and E.~Mandrescu,
\emph{Critical sets, crowns and local maximum independent sets},
\emph{Journal of Global Optimization} \textbf{83} (2022), no.~3, 481--495.

\bibitem{NT1975}
G.~L.~Nemhauser and L.~E.~Trotter, Jr.,
\emph{Vertex packings: structural properties and algorithms},
\emph{Mathematical Programming} \textbf{8} (1975), 232--248.

\bibitem{Pereyra2026}
K.~Pereyra,
\emph{On Graphs with $\CritIndep(G)=\Crown(G)$},
HAL preprint hal-05557006v1, 2026.

\end{thebibliography}
\end{document}